\documentclass[twoside,a4paper,12pt]{amsart}
\usepackage{amssymb}
\usepackage{amsmath}
\usepackage{harpoon}

\usepackage[euler]{textgreek}
\usepackage[all,color]{xy}

\usepackage{ifpdf}
\ifpdf
  \usepackage[colorlinks,linkcolor=red,anchorcolor=blue,citecolor=green]{hyperref}
\else
\fi

\def\isdraft{1}

\ifx\isdraft\undefined
\else
    \usepackage{fancyhdr}
\fi

\newtheorem{theorem}{Theorem}[section]

\newtheorem{corollary}[theorem]{Corollary}
\newtheorem{lemma}[theorem]{Lemma}

\newtheorem{proposition}[theorem]{Proposition}

\newtheorem{fact}[theorem]{Fact}

\theoremstyle{definition}
\newtheorem{definition}[theorem]{Definition}

\numberwithin{equation}{section}

\renewcommand{\vec}{\overrightarrow}
\renewcommand{\hat}{\widehat}
\newcommand{\N}{{\mathbb{N}}}

\newcommand{\Cod}{\operatorname{Cod}}
\newcommand{\onomega}{\omega}

\newcommand{\RCA}{\mathsf{RCA}_0}
\newcommand{\WKL}{\mathsf{WKL}_0}

\newcommand{\RT}{\mathsf{RT}}

\newcommand{\TT}{\mathsf{TT}}
\newcommand{\M}{\mathfrak{M}}

\title[The Infinite Pigeonhole Principle for Trees]{Conservation Strength of The Infinite Pigeonhole Principle for Trees}

\author[Chong]{C.~T.~Chong}
\address{Department of Mathematics\\
National University of Singapore\\Singapore 119076}
\email{chongct@nus.edu.sg}

\author[Wang]{Wei Wang}
\address{Institute of Logic and Cognition and Department of Philosophy\\Sun Yat-Sen University\\Guangzhou, China}
\email{wwang.cn@gmail.com}

\author[Yang]{Yue Yang}
\address{Department of Mathematics\\National University of Singapore\\Singapore 119076}
\email{matyangy@nus.edu.sg}

\thanks{Chong's research was partially supported by NUS grants C-146-000-042-001 and WBS : R389-000-040-101.
Wang's research was partially supported by China NSF Grant 11971501.
Yang's research was partially supported by MOE2019-T2-2-121.
All the authors acknowledge the support of the NUS Institute for Mathematical Sciences where the work began.}


\begin{document}

\begin{abstract}
Let $\TT^1$ be the combinatorial principle stating that every finite coloring of the infinite full binary tree has a homogeneous isomorphic subtree. Let $\RT^2_2$ and $\mathsf{WKL}_0$ denote respectively the principles of Ramsey's theorem for pairs and weak K\"onig's lemma.
It is proved that $\TT^1+\RT^2_2+\mathsf{WKL}_0$ is $\Pi^0_3$-conservative over the base system $\RCA$.
Thus over $\RCA$, $\TT^1$ and Ramsey's theorem for pairs prove the same $\Pi^0_3$-sentences.
\end{abstract}

\maketitle

\section{Introduction} \label{s:introduction}
Let $\RCA$ be the system in the language of second-order arithmetic consisting of the usual arithmetic laws, 
 the recursive comprehension axiom, and the $\Sigma_1$-induction scheme.
The main objective of reverse mathematics is to investigate the proof-theoretic strength of a mathematical theorem 
 formulated in this language over a base system such as $\RCA$.
This approach has proved to be very fruitful as can be seen from numerous examples in the literature.
Of particular relevance to the subject of this paper are weak K\"onig's lemma ($\mathsf{WKL}_0$) and Ramsey's theorem for pairs ($\RT^2_2$), 
 two theorems which, when coupled with $\RCA$, 
 constitute major subsystems of second-order arithmetic. $\mathsf{WKL}_0$ states that every infinite binary tree has an infinite path.
Its proof-theoretic strength was shown by Harrington (unpublished, see Simpson \cite{Sim2009}) to be $\Pi^1_1$-conservative over $\RCA$,
 i.e.~every $\Pi^1_1$-sentence provable in $\mathsf{WKL}_0$ is already provable in $\RCA$.
$\RT^2_2$ states that for every two-coloring of pairs of natural numbers, 
 there is an infinite set in which all pairs of numbers have the same color.
It was shown by Hirst \cite{Hirst1987} that $\RCA+\RT^2_2$ implies the $\Sigma_2$-bounding scheme $B\Sigma_2$ 
 whose strength lies strictly between $\Sigma_1$- and $\Sigma_2$-induction (Paris and Kirby \cite{KP1978}).
Principles related to, or inspired by, 
 $\RT^2_2$ are arguably the most intensively studied combinatorial principles in reverse mathematics in the past 25 years.
It is a long standing open problem whether, over $\RCA$, $\RT^2_2$ is $\Pi^1_1$-conservative over $B\Sigma_2$.
A significant step towards resolving the open problem was taken by Patey and Yokoyama \cite{PY} 
 who showed that $\RT^2_2+\mathsf{WKL}_0$ is $\Pi^0_3$-conservative over $\RCA$.
This result motivated the study of the corresponding problem, discussed in this paper, on finite coloring of the full binary tree.

Let $\TT^1$ be the principle stating 
 that every finite coloring of nodes in the full infinite binary tree has an isomorphic subtree that is homogeneous (or monochromatic),
 i.e.~a tree in which all nodes have the same color. 
$\TT^1$ is a deceptively simple principle: Given an instance of $\TT^1$ in a standard model of $\RCA$,
 i.e.~one whose first order part is the set of standard natural numbers, 
 a straightforward argument shows 
  that above some fixed node there is a dense set of nodes computable in the coloring (viewed as a function on the full binary tree)
  with the same color, and this offers an immediate solution to the instance.
However, as proved in Corduan, Groszek and Mileti \cite{CGM2009}, 
 the picture is very different for $\TT^1$ in a model of $\RCA$ in which $\Sigma_2$-induction ($I\Sigma_2$) does not hold.

Recent studies show that $\TT^1$ shares an important property with $\RT^2_2$ (Chong, Li, Wang and Yang \cite{CLWY}, Chong, Slaman and Yang \cite{CSY2017}):
Over $\RCA$, the inductive strength of $\TT^1$ and $\RT^2_2$ each lies strictly below $I\Sigma_2$. 
In fact, by combining  the constructions in \cite{CLWY} and \cite{CSY2017}, one can derive the same result  for 
$\TT^1+\RT^2_2+\mathsf{WKL}_0$, although the  two coloring principles  are long known to imply $B\Sigma_2$.
On the other hand, it was shown in \cite{CLWY} that $\RCA+\TT^1$ is $\Pi^1_1$-conservative over $P\Sigma_1+B\Sigma_2$, 
 where $P\Sigma_1$ is a $\Pi^0_3$-sentence shown by Kreuzer and Yokoyama \cite{KY} to be equivalent to a number of principles,  
 including the totality of the Ackermann function as well as the bounded monotone enumeration principle $\mathsf{BME}_1$,
 introduced in Chong, Slaman and Yang \cite{CSY2014} for the separation of stable Ramsey's theorem principle $\mathsf{SRT}^2_2$ from $\RT^2_2$.
It is not known if $\RCA+\RT^2_2$ is $\Pi^1_1$-conservative over $P\Sigma_1+B\Sigma_2$.

In this paper, we  show in Theorem \ref{thm:conservation-1}  that $\TT^1+\RT^2_2+\mathsf{WKL}_0$ is $\Pi^0_3$-conservative over $\RCA$.
This exhibit yet another common property  shared by  $\TT^1$ and $\RT^2_2$.
It follows (Corollary \ref{cor:TT1-PSigma1}) that this system does not imply $P\Sigma_1$.

Apart from being of independent interest,  Theorem \ref{thm:conservation-1}  also sheds light on the problem comparing the
strengths of $\TT^1$ and $\RT^2_2$ over $\RCA$.
First, since $\TT^1$ holds  in every model of $\RCA+I\Sigma_2$, 
 one concludes  immediately that $\RCA+\TT^1\not\rightarrow\RT^2_2$. 
The converse  is not known.  One may investigate the problem  from the angle of relative conservation strength. 
Thus, if $\TT^1$ is not $\Pi^0_3$-conservative over $\RCA$, then it would follow that $\RT^2_2$ does not imply $\TT^1$ over $\RCA$. 
Theorem \ref{thm:conservation-1} shows that this approach does not resolve the problem.

Our strategy  to establish Theorem \ref{thm:conservation-1} 
  is broadly speaking  similar in outline  to that in  Ko{\l}odziejczyk and Yokoyama \cite{KoYo}, although the details are naturally different.
In Section \ref{s:persistence}, we prove a combinatorial result (Theorem \ref{thm:large-prehomogeneous-trees})
 which states that if a ``very large'' finite tree is finitely colored then it contains an almost homogeneous subtree
which is ``sufficiently large''.
The combinatorial theorem is then applied in Section \ref{s:conservation} to prove the conservation result.

We conclude this section by fixing or recalling some standard notations.
Given a set $X$, let $[X]^n$ denote the collection of $n$-element subsets of $X$.
When $n = 2$, elements of $[X]^2$ are identified with ordered pairs $(x,y)$ where $x < y$.
Greek letters $\rho, \sigma, \tau, \eta, \ldots$ denote finite binary strings.
Given $\sigma$ and $\tau$, write $\sigma \preceq \tau$ (respectively $\sigma \prec \tau$) if $\sigma$ is a (respectively proper) initial segment of $\tau$.
The length of $\sigma$ is denoted  $|\sigma|$, and the initial segment of $\sigma$ of length $n$ (if it exists) is denoted $\sigma \upharpoonright n$.
Two strings are \emph{compatible} if one is an initial segment of the other.  A \emph{tree} is a set of finite strings closed under initial segments.
A subset $S$ is \emph{compatible} with $\sigma$ if all members of $S$ are compatible with $\sigma$.
By abuse of notation, we write $2^i$  for both $2^i$ as a natural number and as the set of all strings of length $i$, which will be clear from the context to which object the notation refers.

Subsets of $\mathbb{N}$ are denoted $X, Y, Z, X_i, Y_i, Z_i$ etc.  Very often, we will restrict our attention to a finite section $\{\sigma\in 2^{<\omega}: |\sigma|\in X\}$ of a ``large'' finite set $X$, large
in the sense of Ketonen and Solovay \cite{KS1981}.
Recall that a nonempty set $X$ is \emph{$\omega$-large} if $|X| > \min X$.
Given $x$, $X$ and $Y$, write $X < Y$ if   $X = \emptyset \neq Y$ or $\max X < \min Y$, and write $x < X$ (respectively $X < x$) if $\{x\} < X$ (respectively $X < \{x\}$).
Given $n\in\mathbb{N}$, a \emph{stack} of sets $\{X_m: m < n\}$ (of size $n$) is a collection of nonempty finite sets such that $X_m < X_{m+1}$ for $m < n-1$.
We say that $X$ is \emph{$\omega^{d} \cdot n$-large} if $X$ is the union of a stack of $n$ many $\omega^d$-large sets;
and $X$ is \emph{$\omega^{d+1}$-large} if $X = \{\min X\} \cup X_1$ where $\min X_1 > \min X$ and $X_1$ is $\omega^d \cdot \min X$-large.
The following property of $\alpha$-largeness will be used, sometimes implicitly, in the subsequent discussion:

\begin{fact} \label{fac:large-inj}
If $X$ is $\omega^n$-large and $f: X \to Y$ is an injection such that $f(x) \leq x$ for all $x$, then $Y$ is also $\omega^n$-large.

In particular, a superset of an $\omega^n$-large set is also $\omega^n$-large.
\end{fact}

Let $\exp(x) = 2^x$ and $\exp^{n+1}(x) = \exp(\exp^n(x))$.
The following useful bound is given in \cite[\S 2]{KoYo}.

\begin{fact}\label{fac:omega3-large}
If $X$ is $\onomega^3$-large then $|X| > \exp^{\min X}(\min X)$.
\end{fact}

In Section \ref{s:persistence}, we introduce two  notions of largeness for finite trees: persistence and superpersistence. 
A tree with either of these properties is large in the sense that, even after a thinning operation, 
 the resulting tree continues to contain a subtree that is homogeneous for a given coloring and of sufficiently large size.

\section{Persistence and a key combinatorial theorem} \label{s:persistence}

This section is devoted to proving Theorem \ref{thm:large-prehomogeneous-trees} 
 which is used in an essential way to establish our main result in the next section. 
The proof of Theorem \ref{thm:large-prehomogeneous-trees} applies a number of combinatorial lemmata concerning a largeness notion for trees called {\it persistence} and its generalization {\it superpersistence}.

For a model $\M\models \RCA$, we let $\mathbb{N}$ denote the first-order part of $\M$ and call it the set of natural numbers (in the model).
Reference to standard natural numbers will be explicitly stated to avoid any possible ambiguity.
The reader familiar with fragments of Peano arithmetic will observe that the discussion in this section can be carried out in any model of $\RCA$.

We begin with defining some notions concerning trees.

\begin{definition} \label{def:q-strong-tree}
Let $X= \{x_0 < \ldots < x_n\}$ be a subset of $\mathbb{N}$. A finite binary tree $T$ is \emph{$X$-quasistrong}  if
\begin{enumerate}
 \item [(i)] $T \cap 2^{x_i} \neq \emptyset$ for every $x_i \in X$;
 \item [(ii)] for each $i <n$, every $\sigma \in T \cap 2^{x_i}$ has exactly two incompatible extensions in $T \cap 2^{x_{i+1}}$.
\end{enumerate}
\end{definition}

Intuitively, the nodes on an $X$-quasistrong tree are required to split in a well-controlled manner determined by  $X$.
More precisely, every node of length $x_i$ has incompatible extensions of length $x_{i+1}$.
This allows one to measure the largeness of such trees in terms of the $\alpha$-largeness of $X$.

\begin{definition} \label{def:prehomogeneous}
Given $X = \{x_0 < \ldots < x_n\} \subseteq \mathbb{N}$ and a coloring $C$ of $2^{<\mathbb{N}}$, a tree $T \subseteq 2^{<\mathbb{N}}$ is \emph{$(C,X)$-prehomogeneous for color $c$}
if for every $i<n$, $\sigma \in T \cap 2^{x_i}$ and $\sigma \prec \tau \in T \cap 2^{x_{i+1}}$, there exists a $\zeta\in T$ such that $\sigma \preceq \zeta \preceq \tau$ and $C(\zeta) = c$.
A tree is \emph{$(C,X)$-prehomogeneous} if it is $(C,X)$-prehomogeneous for some color.
\end{definition}

The main combinatorial theorem states:

\begin{theorem} \label{thm:large-prehomogeneous-trees}
Let $d,k\in\mathbb{N}$. If $X$ is an $\onomega^{2d+1}$-large finite set with $\min X > \max\{4,k\}$ and $C$ is a $k$-coloring of an $X$-quasistrong finite binary tree $T$,
then there exist an $\omega^d$-large $Y \subseteq X$ and an $S \subseteq T$ such that $S$ is both $Y$-quasistrong and $(C,Y)$-prehomogeneous.
\end{theorem}

The proof of the theorem requires the use of two types   of  tree thinning operation.
The notion of persistence is related to the first type of thinning which is induced by coloring of the leaves.

\begin{definition} \label{def:persistent}
Let $n, d, k\ge 1$, $\alpha=\onomega^d \cdot n$ and let $X$ be a nonempty finite subset of $\N$. Then
\begin{enumerate}
\item [(i)] $X$ is \emph{$(\alpha,k,0)$-persistent}  if $X$ is $\alpha$-large;

\item [(ii)]  $X$ is \emph{$(\alpha,k,i)$-persistent}, for $i\geq 1$,
 if $X$ contains an $(\alpha,k,i-1)$-persistent subset $Y$ 
  such that for any $X$-quasistrong tree $T$ and any coloring of the leaves $C: T \cap 2^{\max X} \to k$,
  there exist a $c < k$ and a $Y$-quasistrong finite tree $S \subseteq T$, 
   such that every leaf of $S$ has an extension in $C^{-1}(c)$.
\end{enumerate}
\end{definition}

From now on, $\alpha$ is always of the form $\onomega^d \cdot n$.
The persistence property bears some resemblance to that of $\alpha$-largeness:

\begin{proposition} \label{fac:persistent} \ 
\begin{enumerate}
 \item [(i)] If $k' \geq k$, $i' \geq i$ and $X$ is $(\alpha,k',i')$-persistent, then $X$ is $(\alpha,k,i)$-persistent.
 \item [(ii)]  Every $(\alpha,k,i)$-persistent set is $\alpha$-large.
 \item [(iii)] Every superset of an $(\alpha,k,i)$-persistent set is $(\alpha,k,i)$-persistent.
 \item [(iv)]  Suppose that $X$ is $(\alpha,k,i)$-persistent, $f: X \to Y$ is injective and $f(x) \leq x$ for all $x \in X$.
  Then $Y$ is also $(\alpha,k,i)$-persistent.
\end{enumerate}
\end{proposition}

\begin{proof}
(i) and (ii) are immediate.

(iii). Suppose  $X \supseteq Y$ and $Y$ is $(\alpha,k,i)$-persistent. 
The case where $i = 0$ follows from Fact \ref{fac:large-inj}.
Assume that $i > 0$.  
Let $Z \subseteq Y$ witness $Y$ being $(\alpha,k,i)$-persistent as in Definition \ref{def:persistent}, 
 in particular, $Z$ is $(\alpha,k,i-1)$-persistent.
Let $T$ be an $X$-quasistrong tree and $C: T \cap 2^{\max X} \to k$.
Let $T_Y = T \cap 2^{\leq \max Y}$. 
Then, we may assume that $T_Y$ is $Y$-quasistrong (by deleting some nodes if necessary).
For $\sigma$ a leaf of $T_Y$, let $\zeta(\sigma) \in T \cap 2^{\max X}$ be the leftmost extension of $\sigma$.
Define
$$
  C_Y(\sigma) = C(\zeta(\sigma)).
$$
By the persistence of $Y$, 
 there exist $S$ and $c < k$ such that $S$ is a $Z$-quasistrong subtree of $T_Y$ and every leaf of $S$ has an extension in $C_Y^{-1}(c)$.
By the definition of $C_Y$, every leaf of $S$ has an extension in $C^{-1}(c)$.
Hence $Z$ witnesses $X$ being $(\alpha,k,i)$-persistent.

(iv).  By (iii), we may assume that $f$ is bijective.  
We prove (iv) by induction on $i$.
The case where $i = 0$ follows from Fact \ref{fac:large-inj}.
Below, assume that $i > 0$.

Let $Z \subseteq X$ witness $X$ being $(\alpha,k,i)$-persistent.  
We prove that $f(Z)$ witnesses $Y$ being $(\alpha,k,i)$-persistent.
As $Z$ is $(\alpha,k,i-1)$-persistent, $f(Z)$ is $(\alpha,k,i-1)$-persistent as well by induction hypothesis.
Let $T$ be a $Y$-quasistrong tree. 
As $f$ is bijective, 
 there is an $X$-quasistrong tree $T_X$ and a bijection $g$
  such that $g: T_X \cap \bigcup_{x \in X} 2^x \to T \cap \bigcup_{y \in Y} 2^y$ 
   and $\sigma \prec \tau$ in $T_X$ if and only if $g(\sigma) \prec g(\tau)$ in $T$.
Thus $g$ is an isomorphism in a sense (after we ignore some irrelevant nodes).
For $C: T \cap 2^{\max Y} \to k$, define $C_X: T_X \cap 2^{\max X} \to k$ by
$$
  C_X(\sigma) = C(g(\sigma)).
$$
Since $Z$ witnesses the persistence of $X$, 
 there exist $c < k$ and $S_X$ such that $S_X$ is a $Z$-quasistrong subtree of $T_X$ and each leaf of $S_X$ has an extension in $C_X^{-1}(c)$.
Let
$$
  S = \{\sigma \in T: \sigma \text{ has an extension in } g(S_X)\}.
$$
Then $S$ is $f(Z)$-quasistrong and each of its leaves has an extension in $C^{-1}(c)$.
\end{proof}

To prove the existence of $(\onomega,\cdot,\cdot)$-persistent sets, 
 we begin with a lemma whose proof is inspired by an idea from Ku\u{c}era-G\'{a}cs coding (\cite{Kuc1984} \cite{Gac1986}) in algorithmic randomness.

\begin{lemma} [Ku\u{c}era-G\'{a}cs coding] \label{lem:Kucera-Gacs}
There exists a primitive recursive function $g: \mathbb{N} \times \mathbb{Q} \to \mathbb{N}$ such that
if $n = g(m,\delta) \geq m \geq 1 \geq \delta > 0$ and $X \in [\mathbb{N}]^{\geq n}$,
then there exists a  $Y \in [X]^m$ such that
\begin{enumerate}
 \item [(i)]  $\min Y = \min X$, and
 \item [(ii)] for every $X$-quasistrong tree $T$ and every $A \subseteq T \cap 2^{\max X}$ with $|A| \geq \delta |T \cap 2^{\max X}|$,
 there exists a $Y$-quasistrong tree $S$ such that $S \subseteq T$ and every leaf of $S$ has an extension in $A$.
\end{enumerate}
\end{lemma}

\begin{proof}
Let $e:=e(\delta)$ be the least integer with
\begin{equation}\label{eq:pos-q-strong-e}
  2^{-e} < \delta.
  \end{equation}
Define a sequence $(e_p)_{p\in \mathbb{N}}$ by
$$
  e_0 = 0, \ e_p = \sum_{r = 1}^{p} (e+r) = e_{p-1} + e + p.
$$
Let
\begin{equation} \label{eq:pos-q-strong-f}
  g(m,\delta) = e_{m-1}= \sum_{r=1}^{m-1} (e+r) = (m-1) e + \frac{(m-1)m}{2}.
\end{equation}

Fix $n = g(m,\delta)$ and $X \in [\mathbb{N}]^{\geq n}$.  
We may further assume that $X$ has exactly $n$ elements
$$
  x_0 < x_1 < \ldots < x_{n-1}.
$$
For $i < m$, let $y_i = x_{e_i}$, and let
$$
  Y = \{y_i: i < m\}.
$$

Let $x = \max X$.  Fix an $X$-quasistrong tree $T$ and $A \subseteq T \cap 2^{x}$ such that $|A| \geq \delta |T \cap 2^{x}| > 2^{-e} |T \cap 2^{x}|$.
For each $\rho \in T$, let
$$
  T_{\rho} = \{\sigma \in T: \sigma \text{ is compatible with } \rho\}.
$$
Let $\rho \in T \cap 2^{y_0}$ be such that
$$
  |A \cap T_\rho \cap 2^{x}| > 2^{-e} |T_\rho \cap 2^x|,
$$
such $\rho$ exists because $|A| \geq \delta |T \cap 2^{x}|$ and because of the additivity of the measure.  
Let
$$
  S_0 = \{\sigma: \sigma \preceq \rho\}.
$$
Clearly $S_0$ is $\{y_0\}$-quasistrong.

Suppose that $p < m-1$, $S_p$ is a subtree of $T \cap 2^{\leq y_p}$ such that $S_p$ is $\{y_i: i \leq p\}$-quasistrong, and
\begin{equation}\label{eq:pos-q-strong-T}
  |A \cap T_\sigma \cap 2^{x}| > 2^{-e-p} |T_\sigma \cap 2^{x}|
\end{equation}
for each $\sigma \in S_p \cap 2^{y_p}$.

For each $\sigma \in S_p \cap 2^{y_p}$, there exist incompatible $\tau(\sigma,0)$ and $\tau(\sigma,1) \in T \cap 2^{y_{p+1}}$ such that
$$
  |A \cap T_{\tau(\sigma,i)} \cap 2^{x}| > 2^{-e-p-1} |T_{\tau(\sigma,i)} \cap 2^{x}|.
$$
The reason is that if not, the total number of nodes in $A\cap T_{\sigma}\cap 2^x$ would be at most
\[
t+\frac{2^{e+p+1}-1}{2^{e+p+1}}t=2t-\frac{1}{2^{e+p+1}}t
\]
where $t=|T_{\tau}\cap 2^x|$ for some (or any) $\tau \in T \cap 2^{y_{p+1}}$, contradicting
$$
  |A \cap T_\sigma \cap 2^{x}| > 2^{-e-p} |T_\sigma \cap 2^x|=2t.
$$
Let
$$
  S_{p+1} = \{\rho: \exists \sigma, i(\sigma \in S_p \cap 2^{y_p} \wedge
    i = 0,1 \wedge \rho \preceq \tau(\sigma,i))\}.
$$
Then $S_{p+1}\subseteq T$ is quasistrong with respect to  $\{y_i: i \leq p+1\}$.
It follows that $S = S_{m-1}$ satisfies the requirement.
\end{proof}

Applying the above, we now show that $(\onomega,\cdot,\cdot)$-persistent sets exist.

\begin{lemma}\label{lem:persistent-existence}
Let $g$ be as in Lemma \ref{lem:Kucera-Gacs}.
Define
$$
  \bar{g}(x,k,0) = x+k+1,\ \bar{g}(x,k,i+1) = g(\bar{g}(x,k,i),1/k).
$$
Then for each $i$,
\begin{enumerate}
 \item [(i)]     $\bar{g}(x,k,i) \leq (x+k+i+1)^{2^{i}}$ for  $x \geq k > 0$;

 \item [(ii)]  If $|X|>\bar{g}(\min X, k, i)$, then  $X$ is $(\onomega,k,i)$-persistent.
\end{enumerate}
\end{lemma}

\begin{proof} We prove the lemma by induction on $i$.

(i). First,
$$
  \bar{g}(x,k,0) = x+k+1 \leq (x+k+1)^{2^0}.
$$
Suppose $i>0$ and let $e$ be the least such that $2^{-e} < 1/k$.
Then $e \leq k$; and an easy induction on $i$ shows that $k< \bar{g}(x,k,i)/2$.
Hence
\begin{align*}
  \bar{g}(x,k,i) &= g(\bar{g}(x,k,i-1),1/k) \\
    &= (\bar{g}(x,k,i-1)-1) \frac{2e + \bar{g}(x,k,i-1)}{2} \mbox{\ \ \ by (\ref{eq:pos-q-strong-f})}\\
    &< \bar{g}(x,k,i-1)^{2} \mbox{\ \ \ by $2e\leq 2k<\bar{g}(x,k,i-1)$}\\
    &\le ((x+k+i)^{2^{i-1}} )^2< (x+k+i+1)^{2^{i}}.
\end{align*}

(ii). If $|X| \geq \bar{g}(\min X,k,0) \geq |\min X| + 1$, then $X$ is $\onomega$-large by definition,
and thus $(\onomega,k,0)$-persistent.

Suppose that $i>0$ and $|X| \geq \bar{g}(\min X,k,i)$.  Let $m = \bar{g}(\min X,k,i-1)$.
Let $Y \in [X]^m$ be as in the conclusion of Lemma \ref{lem:Kucera-Gacs}.
Note that $m = \bar{g}(\min Y,k,i-1)$ since $\min Y = \min X$.
By induction hypothesis, $Y$ is $(\onomega,k,i-1)$-persistent.
Let $T$ be an $X$-quasistrong tree and let $C: T \cap 2^{\max X} \to k$ be a $k$-coloring for some positive $k$.
Choose $c < k$ such that
$$
  |C^{-1}(c)| \geq |T \cap 2^{\max X}|/k.
$$
By Lemma \ref{lem:Kucera-Gacs}, there exists a $Y$-quasistrong tree $S \subseteq T$ whose leaves have extensions in $C^{-1}(c)$.
Thus $X$ is $(\onomega,k,i)$-persistent.
\end{proof}

Next, we construct $(\omega^{d+1},\cdot,\cdot)$-persistent sets through the process of stacking.
We introduce a notion   generalizing   that of quasistrongness.

\begin{definition} \label{def:stack-qstrong}
For a stack $\vec{X} = (X_m: m < n)$, a tree is \emph{$\vec{X}$-quasistrong} if it is $X_m$-quasistrong for each $m < n$.
\end{definition}

Note that an $\vec{X}$-quasistrong tree $T$ may not be $\bigcup \vec{X}$-quasistrong
since an element of $T \cap 2^{\max X_0}$ may not have distinct extensions in $2^{\min X_1}$.
However, the lack of splits between $X_i$ and $X_{i+1}$ is the only missing ingredient.

\begin{lemma} \label{lem:persistent-stack-naive}
Suppose that $i > 0$, $\vec{X} = (X_m: m < n)$ is a stack of $(\alpha,k,i)$-persistent sets and the persistence of each $X_m$ is witnessed by $Y_m$.
If $T$ is $\vec{X}$-quasistrong and $C: 2^{\max X_{n-1}} \to k$, then there exist a $c < k$ and a tree $S \subseteq T$ such that
$S$ is $(Y_m: m < n)$-quasistrong and every leaf of $S$ has an extension in $C^{-1}(c)$.
\end{lemma}

\begin{proof}
We may assume that $n > 0$. Let $X_{-1} = \{0\}$.
If $0 \leq m < n$ and $\rho \in 2^{\max X_{m-1}}$. Let
$$
  T_\rho (=T_{\rho}^m):= \{\sigma \in T: |\sigma| \leq \max X_m\wedge  \sigma \text{ is compatible with } \rho\}.
$$
Note that $T_\rho$ is $X_m$-quasistrong.

For $0\le m\le n-1$, define in succession colorings $C_m$. Let $C_{n-1} = C$.
Suppose that $C_m: 2^{\max X_m} \to k$ is defined and $m-1\ge 0$.
As $Y_m$ witnesses the persistence of $X_m$ and each $T_\rho$ with $\rho \in 2^{\max X_{m-1}}$ is $X_m$-quasistrong,
we can select $c_\rho < k$ and a tree $U_\rho \subseteq T_\rho$ such that $U_\rho$ is $Y_m$-quasistrong and
every leaf of $U_\rho$ has an extension in $C_m^{-1}(c_\rho)$.
Define $C_{m-1}: 2^{\max X_{m-1}} \to k$ by $C_{m-1}(\rho) = c_\rho$.  Define a sequence of trees $\{S_m: m<n\}$ as follows:

Let $c = c_{\emptyset}$, and let $S_0 = U_{\emptyset}$.
For $m < n$, suppose that
\begin{itemize}
 \item $S_m$ is a subtree of $T \cap 2^{\leq \max X_m}$;
 \item $S_m$ is $(Y_\ell: \ell \leq m)$-quasistrong;
 \item every leaf of $S_m$ has an extension in $C_m^{-1}(c)$.
\end{itemize}

Suppose that $m < n-1$. By the definition of $C_m$, for each leaf $\sigma$ of $S_m$,
we can select $\rho(\sigma) \in 2^{\max X_{m}}$ such that $\sigma \preceq \rho(\sigma)$ and $c_{\rho(\sigma)} = c$.
Let $S_{m+1}$ be the union of $S_m$ and the trees $U_{\rho(\sigma)}$'s for $\sigma$ ranging over the leaves of $S_m$. Then $S=S_{n-1}$ is the required tree.
\end{proof}

To obtain an $(\alpha \cdot n, k, i)$-persistent set by the process of stacking,
we may need more than $n$-many $(\alpha,k,i)$-persistent sets to succeed.
The additional  stacks are used to supply the missing splits mentioned after Definition \ref{def:stack-qstrong}.
Roughly speaking, we keep the even layers in the stack and discard  the odd layers to produce the splits,
which reduces the size by a factor of two. The next lemma provides a sufficient condition.

\begin{lemma} \label{lem:persistent-stack}
Suppose that $(X_m: m < 2^i n - 2^i + 1)$ is a stack of $(\alpha,k,i)$-persistent sets.
Then $ \bigcup_m X_m$ is $(\alpha \cdot n,k,i)$-persistent.
\end{lemma}

\begin{proof}
For $i = 0$, $(\alpha \cdot n,k,i)$-persistence and $\alpha \cdot n$-largeness are the same, and the conclusion holds trivially.
Hence assume that $i > 0$.

Let $\bar{n} = 2^i n - 2^i + 1$ and $X = \bigcup_{m < \bar{n}} X_m$.
For each $m$, let $Y_m$ be a subset of $X_m$ witnessing the persistence of $X_m$.
Thus, each $Y_m$ is $(\alpha,k,i-1)$-persistent.
By induction hypothesis, $Y = \bigcup_{2\ell < \bar{n}} Y_{2\ell}$ is an $(\alpha \cdot n, k, i-1)$-persistent subset of $X$.

Let $C: 2^{\max X} \to k$ and let $T$ be an $X$-quasistrong tree.  Then $T$ is $(X_m: m < \bar{n})$-quasistrong.
By Lemma \ref{lem:persistent-stack-naive}, there exist $c < k$ and a tree $S' \subseteq T$ such that
$S'$ is $(Y_m: m < \bar{n})$-quasistrong and every leaf of $S'$ has an extension in $C^{-1}(c)$.

Let $S_0 = S' \cap 2^{\leq \max Y_0}$, which is $Y_0$-quasistrong.
Suppose that $2m < \bar{n}$ and $S_m$ is a $\bigcup_{\ell \leq m} Y_{2\ell}$-quasistrong subtree of $S' \cap 2^{\leq \max Y_{2m}}$.

If $2m = \bar{n}-1$ then let $S = S_m$.

Suppose that $2m < \bar{n}-1$.  As $\bar{n}$ is odd, $2m + 1 < \bar{n}-1$ as well.
If $\sigma$ is a leaf of $S_m$, $|\sigma| = \max Y_{2m}$.
For each such $\sigma$, select two distinct extensions of $\sigma$ in $S' \cap 2^{\max Y_{2m+1}}$, and denote them by $\rho(\sigma,0)$ and $\rho(\sigma,1)$.
The $\rho(\sigma,j)$'s exist because of $S'$ being $Y_{2m+1}$-quasistrong.
Let
\[
  S_{m+1} = \{\tau \in S': |\tau| \leq \max Y_{2m+2}, \tau \text{ is compatible with some } \rho(\sigma,j)\}.
\]
Then $S_{\bar{n}-1}$ is the $Y$-quasistrong subtree that we wanted.
\end{proof}

The lemma below allows us to construct $(\onomega^{d+1},\cdot, \cdot)$-persistent sets by stacking a sufficient number of $(\onomega^{d},\cdot, \cdot)$-persistent sets.

\begin{lemma} \label{lem:persistent-induction}
If  $X$ is $(\onomega^d \cdot (\min X + 1), k, i)$-persistent, then it is $(\onomega^{d+1},k,i)$-persistent.
\end{lemma}

\begin{proof}
If $X$ is $(\onomega^d \cdot (\min X + 1), k, 0)$-persistent, then $X$ is $\onomega^d \cdot (\min X + 1)$-large and thus $\onomega^{d+1}$-large.
So $X$ is $(\onomega^{d+1},k,0)$-persistent.

Suppose that $i > 0$ and $X$ is $(\onomega^d \cdot (\min X + 1), k, i)$-persistent.
Let $Y \subseteq X$ witness the persistence of $X$.  Then $Y$ is $(\onomega^d \cdot (\min X+1),k,i-1)$-persistent.
Let
$$
  Z = \{\min X\} \cup (Y - \{\min Y\}).
$$
By Proposition \ref{fac:persistent}, $Z$ is $(\onomega^d \cdot (\min X+1),k,i-1)$-persistent as well.
By the induction hypothesis and the fact that $\min Z = \min X$, $Z$ is $(\onomega^{d+1},k,i-1)$-persistent.

Let $T$ be an $X$-quasistrong tree and $C: T \cap 2^{\max X} \to k$.
As $Y$ witnesses the persistence of $X$, there exist a $c < k$ and an $S$ such that $S$ is a $Y$-quasistrong subtree of $T$ and every leaf of $S$ has an extension in $C^{-1}(c)$.
Let
$$
  y_1 = \min (Y - \{\min Y\}).
$$
By the $Y$-quasistrongness of $S$, we can select two distinct elements in $S \cap 2^{y_1}$, say $\sigma_0$ and $\sigma_1$, which extend a string  in $S \cap 2^{\min Y}$.
Then $\sigma_0$ and $\sigma_1$ extend a common string in $S \cap 2^{\min X}$ as well.
Let
$$
  U = \{\tau \in S: \tau \text{ is compatible with one of } \sigma_0, \sigma_1\}.
$$
Then $U$ is $Z$-quasistrong and has all its leaves extended by elements of $C^{-1}(c)$.
It follows that $Z$ witnesses $X$ being $(\onomega^{d+1},k,i)$-persistent.
\end{proof}

The following which shows a connection between $\alpha$-largeness and persistence is of independent interest.
It will not be used in our subsequent discussion.

\begin{corollary}\label{cor:persistent-largeness}
Suppose that $d > 0$.
If $X$ is $\onomega^{2d+1}$-large and $\min X > \max\{k,i,2\}$,
then $X$ is $(\onomega^d,k,i)$-persistent.
\end{corollary}

\begin{proof}
Suppose that $d = 1$, $X$ is $\onomega^{3}$-large and $x_0 = \min X > \max\{k,i,2\}$.
Then
\[
  X = \{x_0\} \cup X_1 \cup X_2 \cup X_3,
\]
where $x_0 < X_1 < X_2 < X_3$ and the $X_j$'s are $\onomega^2$-large.
For $j = 1,2,3$, let $x_j = \min X_j$.
Straightforward calculations show that $x_2 > 2^{x_1}$ and
\[
  |X_2 \cup X_3| > 2^{(x_2+2) 2^{x_2}} > (4x_2)^{2^{x_2}} \geq \bar{g}(x_2,k,i),
\]
where $\bar{g}$ is as in Lemma \ref{lem:persistent-existence}.
By Lemma \ref{lem:persistent-existence}, $X_2 \cup X_3$ is $(\onomega,k,i)$-persistent.
Thus $X$ is $(\onomega,k,i)$-persistent as well, by Proposition \ref{fac:persistent}.

Now suppose that $d > 1$, $X$ is $\onomega^{2d + 1}$-large and $x_0 = \min X > \max\{k,i,2\}$.
Then  $X = X_0 \cup X_1$, where $X_0$ and $X_1$ are $\onomega^{2d}$-large sets such that $X_0 < X_1$.
Let $n = 2^i (\min X + 1) - 2^i + 1$.
As $\min X_0 = \min X \geq i$, $\max X_0 \geq n$ by an easy calculation and that $2d > 2$.
Then  $X_1$ is the union of the following sets
$$
  \{\min X_1\} < Y_0 < \ldots < Y_{n-1} < Y_{n},
$$
where each $Y_m$ is $\onomega^{2d-1}$-large.
Let $f$ be the map sending every $j+1$-th element of $X_1$ to the $j$-th element of $X$.
For $m < n$,
 each $f(Y_m)$ is $\onomega^{2d-1}$-large with $\min f(Y_m) > \max\{k,i,2\}$,
 and thus $(\onomega^{d-1},k,i)$-persistent;
 and $\min f(Y_0) = \min X_0 = \min X$.
By Lemma \ref{lem:persistent-stack}, $Z = \bigcup_{m < n} f(Y_m)$ is $(\onomega^{d-1} \cdot (\min Z + 1), k, i)$-persistent.
By Lemma \ref{lem:persistent-induction}, $Z$, and thus $X$, is $(\onomega^{d},k,i)$-persistent.
\end{proof}

To incorporate the property of persistence into the construction of prehomogeneous trees, we introduce another persistence notion similar to that of $\alpha$-largeness.
This notion will handle colorings of an entire tree, instead of simply its leaves as in Definition \ref{def:persistent}.

\begin{definition} \label{def:superpersistent}
A set $X$ is \emph{$(\alpha,k,i)$-superpersistent}, 
 if for any collection $\{T_\rho: \rho \in 2^{\min X}\}$ of $X$-quasistrong trees and any $C: 2^{<\mathbb{N}} \to k$,
 there exist $Y \subseteq X$ and $S_\rho\subseteq T_\rho$ for each $\rho$ such that $Y$ is $(\alpha,k,i)$-persistent, 
 and $S_\rho$ is $Y$-quasistrong as well as $(C,Y)$-prehomogeneous.
\end{definition}

Observe that, apart from the obvious difference that superpersistence of $X$ concerns families of trees $\{T_\rho: \rho\in 2^{\text{min }X}\}$  
 while persistence concerns only single trees $T$, 
 there is the additional point that for the former one considers colorings of each $T_\rho$
 rather than just its leaves, which is the case for th latter, i.e.~the leaves of $T$. 

A superpersistent set shares properties similar to those given in Proposition \ref{fac:persistent} for persistent sets.
We state them below without proof:

\begin{proposition} \label{fac:superpersistent} \

\begin{enumerate}
 \item  [(i)] If $k' \geq k$, $i' \geq j$ and $X$ is $(\alpha,k',i')$-superpersistent, then $X$ is $(\alpha,k,i)$-superpersistent.
 \item  [(ii)] Every $(\alpha,k,i)$-superpersistent set is $\alpha$-large.
 \item [(iii)] Every superset of an $(\alpha,k,i)$-superpersistent set is also $(\alpha,k,i)$-superpersistent.
 \item [(iv)] Suppose that $X$ is $(\alpha,k,i)$-superpersistent, $f: X \to Y$ is injective and $f(x) \leq x$ for all $x \in X$.
  Then $Y$ is also $(\alpha,k,i)$-superpersistent.
\end{enumerate}
\end{proposition}

We begin proving the existence of $(\onomega,\cdot, \cdot)$-superpersistent sets with the following two lemmata.

\begin{lemma} \label{lem:intersection}
Let $Y\subset\mathbb{N}$, $\epsilon, \delta\in (0,1)$ be rational, $S$ be a finite set and let $\{R_y: y\in Y\}$ be a family of subsets  of $S$ with $|R_y| \geq \delta |S|$.
For $p \in [S]^2$, let $Y_p = \{y \in Y: p \in [R_y]^2\}$. Let $P = \{p \in [S]^2: |Y_p| \geq \epsilon |Y|\}$. Then
\[
  |P| > \frac{(\delta |S| - 1)^2 - \epsilon |S|^2}{2(1-\epsilon)}.
\]
In particular, if $(\delta - \sqrt{\epsilon}) |S| > 1$ then there is a $p$ such that $|Y_p|>\epsilon|Y|$.
\end{lemma}

\begin{proof}
We do a counting argument.  Consider the set
\[
  Q = \{(p,y): p \in [R_y]^2\}.
\]

On the one hand, as $|R_y| \geq \delta |S|$, there are at least $2^{-1}\delta |S| (\delta |S| - 1)$ many elements in each $[R_y]^2$, thus
\[
  |Q| \geq 2^{-1} \delta |S| (\delta |S| - 1) |Y| > 2^{-1} (\delta |S| - 1)^2 |Y|.
\]
On the other hand, $Q=\{(p,y): p\in [R_y]^2\}=\{(p,y): y\in Y_p\}$, thus
\begin{align*}
  |Q| = \sum_{p \in P} |Y_p| + \sum_{p \not\in P} |Y_p|
    < |P| |Y| + (2^{-1}|S|^2 - |P|) \epsilon |Y|.
\end{align*}
Hence
\[
  |P| > \frac{(\delta |S| - 1)^2 - \epsilon |S|^2}{2(1-\epsilon)}.
\]
In particular, when $(\delta - \sqrt{\epsilon}) |S| > 1$, $P$ is nonempty, the result follows.
\end{proof}
%

Fix a subset $X$ of $\mathbb{N}$.
\begin{lemma}\label{lem:pos-antichains}
There is a primitive recursive function $h: \mathbb{N}^2 \times \mathbb{Q} \to \mathbb{N}$ such that for any $Y$, $\{T_k: k < \ell\}$ and
$\{A_{k,y}: k < \ell, y \in Y\}$ satisfying
\begin{enumerate}
 \item [(a)] $Y$ is a (finite) subset of $X$ and $|Y| \geq h(\ell,m,\delta)$;
 \item [(b)] $\{T_k: k < \ell\}$ is a family of $X$-quasistrong trees and $T_k \cap 2^{\min X}$ has a single element $\rho_k$;
 \item [(c)] $A_{k,y}\subseteq T_k \cap 2^y$ and $|A_{k,y}| \geq \delta |T_k \cap 2^y|$,
\end{enumerate}
there exist  a $Z \in [X]^{m}$ and a family of trees $\{S_k: k < \ell\}$ such that
\begin{enumerate}
 \item [(i)] $\min Z = \min X$;
 \item [(ii)] $S_k$ is a $Z$-quasistrong subtree of $T_k$, and
 \item [(iii)] if $(x,z) \in [Z]^2$ and $\sigma \in S_k \cap 2^z$, then $\sigma \upharpoonright y \in A_{k,y}$ for some $y \in Y \cap [x,z]$.
\end{enumerate}
\end{lemma}

In fact the most important property of $h$ that we will  need is the following: for all $m, \ell>0$, $0<\delta\le 1$ rational,
\begin{equation}\label{eq:pos-antichains}
  h(\ell,m,\delta) \geq h(2\ell, m-1, \delta/4) 2^{3 \ell (e+2)} + e + 3,
\end{equation}
where $e$ is the least such that $2^{-e} < \delta$.

\vskip.2in

\noindent {\it Remark.} The informal idea is that $h(\ell,m,\delta)$  offers  enough space to accommodate a desired $Z$ of size  $m$.   
If $m-1$-many members of $Z$ are determined, 
 then in the course of obtaining the $m^{\rm th}$-member of $Z$, 
    the   number $\ell$ of  leaves of $S_k$ is doubled to $2\ell$ and the density $\delta$ may be reduced by some factor, say $1/4$.
In the above inequality, the  numbers   $2^{3 \ell (e+2)}$ and $e+3$ come from combinatorial calculations that  show up in the proof below. 

There exist  primitive recursive functions satisfying \eqref{eq:pos-antichains}. For instance, we may take
\begin{equation}\label{eq:h}
  h(\ell,m,\delta) = \exp\{e \ell \exp(5m)\},
\end{equation}
where $e$ is the least integer such that $2^{-e} < \delta$.
If $\ell$ and $m$ are positive, then
\begin{align*}
  h(2\ell, m-1, \delta/4) 2^{3 \ell (e+2)} &= \exp\{(e+2) \ell \exp(5m-4) + 3 (e+2) \ell\} \\
    &\leq \exp\{e \ell (\exp(5m-2) + 9)\} \\
    &\leq h(\ell,m,\delta) - (e + 3).
\end{align*}
Hence  $h$ satisfies \eqref{eq:pos-antichains}.

\begin{proof}
Suppose that $Y$, $\{T_k\}$ and $\{A_{k,y}\}$ satisfy (a)--(c).  Then $|Y| \geq h(\ell,m,\delta) > e+3$.
Let $Y_0$ be obtained by removing the first $e+3$ many elements from $Y$, and let $y_0 = \min Y_0$.
Clearly, by \eqref{eq:pos-antichains}
\begin{equation}\label{eq:pos-antichains-Y0}
  |Y_0| \geq h(\ell, m, \delta)-(e+3)\geq h(2\ell,m-1,\delta/4) 2^{3\ell(e+2)}.
\end{equation}
Note that by the quasistrongness of $T_k$,
\begin{equation}\label{eq:pos-antichains-w}
  y \in Y_0 \to |T_k \cap 2^{y}| \geq 2^{e+3}.
\end{equation}
For each $y \in Y_0$, let
$$
  B_{k,y} = \{\sigma \in T_k \cap 2^y: \exists x \in Y_0 \cap [0,y] (\sigma \upharpoonright x \in A_{k,x})\}.
$$
We chop the (density) interval $[0,1]$ into $2^{e+2}$ many subintervals and consider the interval $[\frac{v}{2^{e+2}}, \frac{v+1}{2^{e+2}})$ 
 that $\frac{|B_{k,y}|}{|T_k \cap 2^y|}$ falls into and define:
$$
  f_y: \ell \to 2^{e+2}, \ f_y(k) = \max\left\{v < 2^{e+2}: |B_{k,y}| \geq \frac{v}{2^{e+2}} |T_k \cap 2^y| \right\}.
$$
Since there are $2^{\ell(e+2)}$ many such $f_y$'s, by \eqref{eq:pos-antichains-Y0} and the  Pigeonhole Principle
we can select $Y_1 \subseteq Y_0$  such that
\begin{equation}\label{eq:pos-antichains-Y1-size}
  |Y_1| \geq h(2\ell, m-1, \delta/4) 2^{2\ell(e+2)},
\end{equation}
and for all $y,y' \in Y_1$, $f_{y} = f_{y'}$.
Let $y_1 = \min Y_1$ and $B_k = B_{k,y_1}$.
Note that for each $y \in Y_1$ and $k < \ell$,
\begin{equation}\label{eq:pos-antichains-Y1}
  |\{\sigma \in A_{k,y}: \sigma \upharpoonright y_1 \in B_{k}\}| > (\delta - 2^{-e-2}) |T_k \cap 2^y|
    > \frac{3\delta}{4} |T_k \cap 2^y|.
\end{equation}
The first inequality uses the fact that
\[
|\{\sigma\in A_{k,y}: \sigma \upharpoonright y_1 \not\in B_{k}\}|<\frac{1}{2^{e+2}}|T_k \cap 2^y|,
\]
since otherwise, the number $\frac{|B_{k,y}|}{|T_k \cap 2^y|}$ would fall into some interval $[\frac{v'}{2^{e+2}}, \frac{v'+1}{2^{e+2}})$
with $v'>v$, contradicting  the definition of $Y_1$.
For each $\zeta \in T_k \cap 2^{y_1}$, let
$$
  T_{k,\zeta} = \{\sigma \in T_k: \sigma \text{ is compatible with } \zeta\},
$$
and let
$$
  A_{k,\zeta,y} = A_{k,y} \cap T_{k,\zeta}.
$$
For $k < \ell$ and $y \in Y_1$, let
\[
  D_{k,y} = \left\{ \zeta \in B_k: |A_{k,\zeta,y}| \geq \frac{\delta}{4} |T_{k,\zeta} \cap 2^y| \right\}.
\]
Then for $y \in Y_1$,
\begin{align*}
  |\{\sigma \in A_{k,y}: \sigma \upharpoonright y_1 \in B_{k}\}| &=
    \big| \bigcup_{\zeta \in D_{k,y}} A_{k,\zeta,y} \big| + \big| \bigcup_{\zeta \in B_k - D_{k,y}} A_{k,\zeta,y} \big| \\
    &< \left( \frac{|D_{k,y}|}{|T_k \cap 2^{y_1}|} + \frac{\delta}{4} \right) |T_k \cap 2^y|.
\end{align*}
Putting this together with \eqref{eq:pos-antichains-Y1},
we obtain
\begin{equation}\label{eq:pos-antichains-Dky-size}
  |D_{k,y}| \geq 2^{-1}\delta |T_k \cap 2^{y_1}|.
\end{equation}

By \eqref{eq:pos-antichains-w} and \eqref{eq:pos-antichains-Dky-size}, 
 we can apply Lemma \ref{lem:intersection} to $Y_1$, $\epsilon=\frac{\delta^2}{4^2}$, $\frac{\delta}{2}$,
 $S = T_0 \cap 2^{y_1}$ and the family $(D_{0,y}: y \in Y_1)$.  
Since $(\delta-\sqrt{\epsilon})|S|>1$ because $|S| \geq 2^{y_1} > 2^{e+2}$,
 we obtain $p=(\zeta(0,0),\zeta(0,1))$ and $Z_0$ which is the $Y_p$ in Lemma \ref{lem:intersection} 
  such that $\zeta(0,0)$ and $\zeta(0,1)$ are distinct elements of $T_0 \cap 2^{y_1}$,
  $\zeta(0,i) \in D_{0,y}$ for $y \in Z_0 \subseteq Y_1$ and $|Z_0| \geq 2^{-4} \delta^2 |Y_1|$.
By $(\ell-1)$-many inductive applications of Lemma \ref{lem:intersection}, 
 from $T_1$ to $T_{\ell-1}$ consecutively,
 we have $Z_\ell \subseteq Z_0 \subseteq Y_1$ and $(\zeta(k,i): k < \ell, i < 2)$ 
  such that $\zeta(k,0)$ and $\zeta(k,1)$ are distinct elements of $T_k \cap 2^{y_1}$,
  $\zeta(k,i) \in D_{k,y}$ for $y \in Z_\ell$ and $|Z_\ell| \geq 2^{-4\ell} \delta^{2\ell} |Y_1|$.

Let $X' = X \cap [y_1,\infty)$, $Y' = Z_\ell$.
By \eqref{eq:pos-antichains-Y1-size} and $2^{-e} < \delta$,
\[
  |Y'| \geq h(2\ell,m-1,\delta/4).
\]
Applying the induction hypothesis to $X',Y'$, the $T_{k,\zeta(k,i)}$'s and $A_{k,\zeta(k,i),y}$'s,
we conclude that there exist $Z' \in [X']^{m-1}$ and $S_{k,\zeta(k,i)}$ ($i= 0,1$) such that
\begin{itemize}
\item  $\min Z' = \min X' = y_1$;
\item  $S_{k,\zeta(k,i)}$ is a $Z'$-quasistrong subtree of $T_{k,\zeta(k,i)}$;
\item  If $(x,z) \in [Z']^2$ and $\sigma \in S_{k,\zeta(k,i)}$ then $\sigma \upharpoonright y \in A_{k,\zeta(k,i),y}$ for some $y \in Y' \cap [x,z]$.
\end{itemize}

Finally, let $Z = \{\min X\} \cup Z' \in [X]^{m}$,
and for each $k < \ell$, let
$$
  S_k = \{\sigma: \exists i < 2 (\sigma \prec \zeta(k,i))\} \cup S_{k,\zeta(k,0)} \cup S_{k,\zeta(k,1)}.
$$
It is straightforward to verify that $Z$ and the $S_k$'s satisfy conditions (ii) and (iii).
\end{proof}

%

\begin{lemma} \label{lem:prehom-multiple}
There is a primitive recursion function $\bar{h}: \mathbb{N}^3 \to \mathbb{N}$ such that
for all $X$,  $(T_\rho: \rho \in 2^{\min X})$ and $C$ such that 
\begin{enumerate}
 \item[(a)] $X$ is a finite set with $|X| \geq \bar{h}(\min X, n, k)$,
 \item[(b)] $(T_\rho: \rho \in 2^{\min X})$ is a collection of trees where each $T_\rho$ is $X$-quasistrong and compatible with $\rho$,
 \item[(c)] $C: \bigcup_\rho T_\rho \to k$ is a coloring,
\end{enumerate}
 there exist a $Z \in [X]^n$ and a family $(S_\rho \subseteq T_\rho:\rho \in 2^{\min X})$ satisfying
 \begin{enumerate}
 \item [(i)] $\min Z = \min X$;
\item [(ii)] Each $S_\rho$ is $Z$-quasistrong and $(C,Z)$-prehomogeneous.
\end{enumerate}
\end{lemma}

\begin{proof} We verify that
$$
  \bar{h}(x,n,k) = h(2^x,n,1/k) k^{2^x}
$$
works.
For each $x \in X$, define $f_x: 2^{\min X} \to k$ as follows:
$$
  f_x(\rho) = \min \left\{c < k: |T_\rho \cap 2^x \cap C^{-1}(c)| \geq \frac{|T_\rho \cap 2^x|}{k} \right\}.
$$
By (a) and the Pigeonhole Principle, we can select $Y \subseteq X$ such that $|Y| \geq h(2^{\min X},n,1/k)$ and $f_x = f_y$ for all $(x,y) \in [Y]^2$.
Let $f = f_y$ for any $y \in Y$.

For each $\rho \in 2^{\min X}$ and $y \in Y$, let
$$
  A_{\rho,y} = T_\rho \cap 2^y \cap C^{-1}(f(\rho)).
$$
It is easy to verify that $X, Y$, and the $T_\rho$'s as well as $A_{\rho,y}$'s satisfy the hypothesis of Lemma \ref{lem:pos-antichains} for $\ell = 2^{\min X}$, $m = n$ and $\delta = 1/k$.
Hence there exist $Z \in [X]^{n}$ and $(S_\rho: \rho \in 2^{\min X})$ satisfying the conclusions of Lemma \ref{lem:pos-antichains}.
Then $Z$ and the $S_\rho$'s are as required.
\end{proof}

We are now ready to show the existence of $(\onomega,\cdot,\cdot)$-superpersistent sets.

\begin{lemma}\label{lem:superpersistent-omega}
If $|X| \geq \bar{h}(\min X, \bar{g}(\min X, k, i), k)$
then $X$ is $(\onomega,k,i)$-superpersistent.
\end{lemma}

\begin{proof}
This is a direct consequence of Lemmata \ref{lem:persistent-existence} and \ref{lem:prehom-multiple}.
\end{proof}

Lemmata \ref{lem:persistent-existence} and \ref{lem:superpersistent-omega} imply the following.

\begin{corollary}\label{cor:superpersistent-omega}
Suppose that $k > 0$, $X$ is $\onomega^{3}$-large and $\min X > \max\{k, i, 4\}$.
Then $X$ is $(\onomega,k,i)$-superpersistent.
\end{corollary}

\begin{proof}
Let $k > 0$ and $X$ be $\onomega^3$-large with $x = \min X > \max\{k,i,4\}$.
Let $e$ be the least such that $2^{-e} > 1/k$. So $e \leq k$.
Clearly, $\bar{h}$ is increasing on the second variable.
By Lemma \ref{lem:persistent-existence},
\begin{align*}
  \bar{h}(x, \bar{g}(x, k, i), k) &\leq h(2^x, (x+k+i+1)^{2^i}, 1/k) k^{2^x} \\
    & \leq \exp \{k 2^x \exp(5(x+k+i+1)^{2^i}) + k 2^x\}.
\end{align*}
Since $x > \max\{k, i, 4\}$,
\begin{align*}
  k 2^x \exp(5(x+k+i+1)^{2^i}) + k 2^x &< x \exp((3x)^{2^x} + x) \\
    &< \exp^2 (x 2^x) < \exp^4(x).
\end{align*}
Thus $\bar{h}(x,\bar{g}(x,k,i),k) < \exp^5(x)$.
By Fact \ref{fac:omega3-large} and $X$ being $\onomega^3$-large,
\[
  |X| \geq \exp^x(x) > \bar{h}(x,\bar{g}(x,k,i),k).
\]
Hence $X$ is $(\onomega,k,i)$-large by Lemma \ref{lem:superpersistent-omega}.
\end{proof}

Next, we climb up the ladder of $(\onomega^e,\cdot,\cdot)$-superpersistence by the process of stacking.
We need the following variation of prehomogeneity regarding stacks.

\begin{definition} \label{def:stack-prehomogeneous}
Let $C$ be a coloring on a tree $T$ and $\vec{X} = (X_m: m < n)$ be a stack.
We say that $T$ is \emph{$(C,\vec{X})$-prehomogeneous}, if it is $(C,X_m)$-prehomogeneous for each $m < n$.
\end{definition}

Note that a $(C,\vec{X})$-prehomogeneous tree $T$ may not be $(C,\bigcup \vec{X})$-prehomogeneous,
since $T$ could be $(C,X_0)$-prehomogeneous with color $c_0$,
but $(C,X_1)$-prehomogeneous with a different color $c_1$.

\begin{lemma} \label{lem:superpersistent-stack-naive}
Suppose that
\begin{enumerate}
 \item[(a)] $\vec{X} = (X_m: m < n)$ is a stack of $(\alpha,k,i+n-1)$-superpersistent sets;
 \item[(b)] $\{T_\rho: \rho \in 2^{\min X_0}\}$ is a collection of $\vec{X}$-quasistrong trees where each $T_\rho$ is compatible with $\rho$;
 \item[(c)] $C$ is a $k$-coloring on $ \bigcup_\rho T_\rho $.
\end{enumerate}
Then there exist a stack $\vec{Y} = (Y_m: m < n)$ and a collection  $\{S_\rho: \rho \in 2^{\min X_0}\}$ such that
\begin{enumerate}
 \item [(i)] $Y_m$ is an $(\alpha,k,i)$-persistent subset of $X_m$,
 \item [(ii)] $S_\rho$ is a $\vec{Y}$-quasistrong and $(C,\vec{Y})$-prehomogeneous subtree of $T_\rho$.
\end{enumerate}
\end{lemma}

\begin{proof}
Let $x_0 = \min X_0$.

If $m < n$ and $\zeta \in T_{\rho} \cap 2^{\min X_m}$, let
$$
  \hat{T}_\zeta = \{\eta \in T_{\rho}: |\eta| \leq \max X_m \wedge \eta \text{ is compatible with } \zeta\}.
$$
As $X_m$ is $(\alpha,k,i+n-1)$-superpersistent for all $m < n$, select $Y_m^{0}$ and $U_\zeta$, where $\zeta \in 2^{\min X_m}$, such that
\begin{itemize}
\item $Y_m^{0}$ is a $(\alpha,k,i+n-1)$-persistent subset of $X_m$;
\item $U_\zeta$ is a $Y_m^0$-quasistrong subtree of $\hat{T}_\zeta$ and
\item $U_\zeta$ is $(C,Y_m^0)$-prehomogeneous with color $c_\zeta$.
\end{itemize}
For $\rho$ with length $x_0$, let $S_\rho^0 = U_\rho$ and $c_\rho^0 = c_\rho$.
Suppose that $m < n$ and we have the following data,
\begin{itemize}
 \item $Y_\ell^{m}$ is an $(\alpha,k,i+n-1-m)$-persistent subset of $X_m$;
 \item $\vec{Y}_{\leq m}^m = (Y_\ell^m: \ell \leq m)$;
 \item For each $\rho \in 2^{x_0}$, $S_\rho^m$ is a subtree of $T_\rho \cap 2^{\leq \max X_m}$;
 \item  $S_\rho^m$ is $\vec{Y}_{\leq m}^m$-quasistrong and $(C,\vec{Y}_{\leq m}^m)$-prehomogeneous.
\end{itemize}

If $m = n-1$, let $Y_\ell = Y_\ell^{n-1}$ and $S_\rho = S_\rho^{n-1}$.
Then $Y_\ell$'s and $S_\rho$'s are as required.

Suppose that $m < n-1$. We construct the $Y_\ell^{m+1}$'s and $S_\rho^{m+1}$'s as follows.

Let $Y_\ell^{m+1}$ witness the $(\alpha,k,i+n-1-m)$-persistence of $Y_\ell^m$.
So $Y_\ell^{m+1}$ is an $(\alpha,k,i+n-1-m-1)$-persistent subset of $X_m$.
For $\sigma$ a leaf of $S_\rho^m$, 
 select $\zeta(\sigma)$ be such that $\sigma \prec \zeta(\sigma) \in T_\rho \cap 2^{\min X_{m+1}}$.
Let
$$
  C_m(\sigma) = c_{\zeta(\sigma)}.
$$
Then $C_m$ is a $k$-coloring on the leaves of $S_\rho^m$'s.
By Lemma \ref{lem:persistent-stack-naive}, for each $S_\rho^m$,
 there exist $c_{\rho}^m < k$ and $\hat{S}_\rho^m$ such that $\hat{S}_\rho^m$ is a subtree of $S_\rho^m$,
 $\hat{S}_\rho^m$ is $(Y_\ell^{m+1}: \ell \leq m)$-quasistrong, 
 and every leaf $\tau$ of $\hat{S}_\rho^m$ has an extension $\sigma(\tau)$ in $C_m^{-1}(c_\rho^m)$.
By the definition of $C_m$, each $\sigma(\tau)$ corresponds to $\zeta(\sigma(\tau))$ 
 such that $|\zeta(\sigma(\tau))| = \min X_{m+1}$ and $c_{\zeta(\sigma(\tau))} = c_\rho^m$.
Let $S_\rho^{m+1}$ be the union of $\hat{S}_\rho^m$ and $U_{\zeta(\sigma(\tau))}$'s, 
 where $\tau$ ranges over the leaves of $\hat{S}_\rho^m$.

Let $\vec{Y}_{\leq m+1}^{m+1} = (Y_\ell^{m+1}: \ell \leq m+1)$.
As each $U_{\zeta(\sigma(\tau))}$ is $Y_{m+1}^{m+1}$-quasistrong, 
 $S_\rho^{m+1}$ is $\vec{Y}_{\leq m+1}^{m+1}$-quasistrong;
 and as each $U_{\zeta(\sigma(\tau))}$ is $(C,Y_{m+1}^{m+1})$-prehomogeneous with color $c_\rho^m$, 
 $S_\rho^{m+1}$ is $(C,\vec{Y}_{\leq m+1}^{m+1})$-prehomogeneous.
\end{proof}

With the above lemma and the pigeonhole principle, 
 we can construct $(\alpha \cdot n,\cdot, \cdot)$-superpersistent sets by stacking $(\alpha, \cdot,\cdot)$-superpersistent sets.

\begin{lemma}\label{lem:superpersistent-stack}
Suppose that $\hat{n} = (2^{i+1} n - 2^{i+1}) k^{2^{x_0}} + 1$ and $j = i + \hat{n} - 1$.
If $(X_m: m < \hat{n})$ is a stack of $(\alpha,k,j)$-superpersistent sets and $x_0 = \min X_0$,
then $X = \bigcup_{m < \hat{n}} X_m$ is $(\alpha \cdot n, k, i)$-superpersistent.
\end{lemma}

\begin{proof}
Fix a family $\{T_\rho: \rho \in 2^{x_0}\}$ of $X$-quasistrong trees and let $C$ be a $k$-coloring.
Then each $T_\rho$ is $(X_m: m < \hat{n})$-quasistrong.

By Lemma \ref{lem:superpersistent-stack-naive}, 
 there exist $\vec{Y} = (Y_m: m < \hat{n})$ and $\hat{S}_\rho$'s such that $Y_m$ is a $(\alpha,k,i)$-persistent subset of $X_m$,
 $\hat{S}_\rho$ is a $\vec{Y}$-quasistrong subtree of $T_\rho$, 
 and $\hat{S}_\rho$ is $(C,\vec{Y})$-prehomogeneous.
Let $c_\rho^m$ be such that if $(x,z) \in [Y_m]^2$ and $\sigma \in \hat{S}_\rho \cap 2^z$ 
 then $C(\sigma \upharpoonright y) = c_\rho^m$ for some $y \in [x,z]$.

For each $m < \hat{n}$ and $\rho\in 2^{x_0}$, we have $c_\rho^m \in k$.
Hence there exist $L \subseteq [0,\hat{n}-1]$ and $\{c_\rho: \rho \in 2^{x_0}\}$ 
 such that $|L| = 2^{i+1} n - 2^{i+1} + 1$ and $c_\rho^\ell = c_\rho$ for each $\ell \in L$ and $\rho \in 2^{x_0}$.
List the members of $L$ as
$$
  m_0 < m_1 < \ldots < m_{2^{i+1} n - 2^{i+1}}.
$$
Let $Z_\ell = Y_{m_{2\ell}}$.
By Lemma \ref{lem:persistent-stack}, 
 $Z = \bigcup_{\ell \leq 2^i n - 2^i} Z_\ell$ is $(\alpha \cdot n,k,i)$-persistent.
Now, it is easy to find for each $\rho$ a tree $S_\rho \subseteq \hat{S}_\rho$, 
 which is $Z$-quasistrong and $(C,Z)$-prehomogeneous with color $c_\rho$.
\end{proof}

Lemma \ref{lem:superpersistent-induction} below is an analog of Lemma \ref{lem:persistent-induction}.

\begin{lemma}\label{lem:superpersistent-induction}
If $X$ is $(\onomega^e \cdot (\min X + 1), k, i)$-superpersistent,
then $X$ is $(\onomega^{e+1},k,i)$-superpersistent.
\end{lemma}

\begin{proof}
Let $\{T_\rho: \rho \in 2^{\min X}\}$ be a family of $X$-quasistrong trees such that $T_\rho \cap 2^{\min X} = \{\rho\}$,
and let $C$ be a $k$-coloring on $\bigcup_\rho T_\rho $.
By the $(\onomega^e \cdot (\min X + 1), k, i)$-superpersistence of $X$, there exist $Z \subseteq X$ and $\{\hat{S}_\rho: \rho \in 2^{\min X}\}$ such that $Z$ is $(\onomega^e \cdot (\min X + 1), k, i)$-persistent,
each $\hat{S}_\rho$ is a $Y$-quasistrong and $(C,Z)$-prehomogeneous subtree of $T_\rho$.
For each $\rho$, select $\sigma(\rho) \in \hat{S}_\rho \cap 2^{\min Z}$.
Let $S_\rho$ be the subtree of $\hat{S}_\rho$ consisting of elements compatible with $\sigma(\rho)$.
Let
$$
  Y = \{\min X\} \cup (Z - \{\min Z\}).
$$
By Proposition \ref{fac:persistent}, $Y$ is $(\onomega^e \cdot (\min X + 1), k, i)$-persistent as well.
By Lemma \ref{lem:persistent-induction} and the fact that $\min Y = \min X$, $Y$ is $(\onomega^{e+1},k,i)$-persistent.
Moreover, the $S_\rho$'s are $Y$-quasistrong and $(C,Y)$-prehomogeneous.
Thus $X$ is $(\onomega^{e+1},k,i)$-superpersistent.
\end{proof}

We now establish a connection between $\alpha$-largeness and superpersistence, similar to that shown in Corollary \ref{cor:persistent-largeness}.

\begin{lemma}\label{lem:superpersistent-large}
If $d > 0$ and $X$ is an $\onomega^{2d+1}$-large set with $\min X > \max\{4,k,i\}$, then $X$ is $(\onomega^{d},k,i)$-superpersistent.
\end{lemma}

\begin{proof}
We prove the lemma by induction on $d$.
The case where $d = 1$ follows from Lemma \ref{lem:superpersistent-omega}.
Suppose that $d > 1$ and $X$ is $\onomega^{2d+1}$-large.
Then $X = X_0 \cup X_1$, where $X_0 < X_1$ and both $X_i$'s are $\onomega^{2d}$-large.
Let $x_0 = \min X$, $n = x_0 + 1$, $\hat{n} = (2^{i+1} n - 2^{i+1}) k^{2^{x_0}} + 1$, and $j = i + \hat{n} - 1$.
By easy calculations and Fact \ref{fac:omega3-large},
$$
  \min X_1 > \max\{k,\hat{n},j\}.
$$
So $X_1 - \{\min X_1\}$ is a union of a stack $(Y_m: m < \hat{n})$ of $\onomega^{2d-1}$-large sets.
By induction hypothesis,  each $Y_m$ is  $(\onomega^{d-1},k,i)$-superpersistent.
By Lemma \ref{lem:superpersistent-stack}, $X_1 - \{\min X_1\}$ is $(\onomega^d \cdot (x_0 + 1),k,i)$-superpersistent.
By Proposition \ref{fac:superpersistent} and Lemma \ref{lem:superpersistent-induction}, $X$ is $(\onomega^{d+1},k,i)$-superpersistent.
\end{proof}

Theorem \ref{thm:large-prehomogeneous-trees} follows immediately from Lemma \ref{lem:superpersistent-large}.

\section{$\Pi^0_3$-Conservation}\label{s:conservation}

In this section we prove the following conservation theorem for $\TT^1$.
Unless otherwise indicated, $\mathbb{N}$ will denote the set of standard natural numbers.

\begin{theorem}\label{thm:conservation}
$\WKL + \TT^1$ is $\Pi^0_3$-conservative over $\RCA$.
\end{theorem}

\begin{proof}
Suppose that $\RCA \not\vdash \forall x \exists y \forall z R(x,y,z)$, where $R$ is a $\Sigma^0_0$-predicate.
We prove that $\RCA + \mathsf{WKL}_0+ \TT^1$ does not imply $\forall x \exists y \forall z R(x,y,z)$ either,
by exhibiting a model of $\RCA + \TT^1$, which does not satisfy $\forall x \exists y \forall z R(x,y,z)$.

Let $\mathfrak M=(M,\mathcal{S})$ be a countable model of $\RCA + \exists x \forall y \exists z \neg R(x,y,z)$,
and let $a \in M$ be such that $\mathfrak M \models \forall y \exists z \neg R(a,y,z)$.
Select an $\M$-infinite $B \in \mathcal{S}$ such that if $(b_0,b_1) \in [B]^2$ then
\begin{equation}\label{eq:conservation-B}
 \mathfrak M \models \forall y < b_0 \exists z < b_1 \neg R(a,y,z).
\end{equation}
Let $X$ be an $\M$-finite subset of $B$ such that $X$ is $\omega^d$-large for some $d \in M \setminus \mathbb{N}$ and $\min X > a$.
By \cite{KoYo} and Theorem \ref{thm:large-prehomogeneous-trees}, we can define a sequence $(X_i: i \in \mathbb{N})$ of  $\M$-finite sets such that
\begin{itemize}
 \item $X = X_0 \supseteq\cdots\supseteq  X_i \supset X_{i+1}\supseteq\cdots$;
 \item In $M$, $X_i$ is $\omega^{d_i}$-large for some $d_i \in M \setminus  \mathbb{N}$;
 \item $\min X_{i+1} > \min X_i$;
 \item If $E$ is an $\M$-finite set with $|E|^{\M} < \min X_i$ then there exists a $j > i$ such that $  [\min X_j, \max X_j] \cap E = \emptyset$;
 \item if $C \in \mathcal{S}$ is a $k$-coloring of $2^{< M}$ for some $k < \min X_i$ then there exist $j > i$ and $S$ such that $S$ is an $\M$-finite $X_j$-quasi-strong and $(C,X_j)$-prehomogeneous tree.
\end{itemize}

Let
$$
  I = \bigcup_{i \in \mathbb{N}} [0,\min X_i].
$$
As in \cite{KoYo}, it is easy to verify that $I$ is a semi-regular cut of $M$, and thus $\mathfrak{N} = (I,\Cod(M/I)) \models \WKL$.
By \eqref{eq:conservation-B} and the fact that $a < \min X_0$,
$$
  \mathfrak{N} \models \exists x \forall y \exists z \neg R(x,y,z).
$$
To see that $\mathfrak{N} \models \TT^1$, fix $k \in I$ and a $k$-coloring $\hat{C}$ of $2^{< I}$.
Then there exist $i \in \mathbb{N}$ and $C \in \mathcal{S}$ such that $k < \min X_i$, $\hat{C} = C \cap I$ and $C$ is a $k$-coloring of $2^{< M}$.
By the construction of the $X_i$'s, let $j > i$ and $S \in M$ be such that $S$ is an $X_j$-quasi-strong and $(C,X_j)$-prehomogeneous tree.
Then $\hat{S} = S \cap I \in \Cod(M/I)$ is a $\hat{C}$-prehomogeneous perfect tree in $\mathfrak{N}$.
As $\mathfrak{N} \models I\Sigma^0_1$, in $\mathfrak{N}$ there exists a $\hat{C}$-homogeneous perfect subtree of $\hat{S}$.
Hence, $\mathfrak{N} \models \TT^1$.
\end{proof}

For readers familiar with \cite{KoYo}, it is not difficult to see that the proof of the above theorem can be combined with the proof of \cite[Theorem 3.3]{KoYo} to yield a stronger result:

\begin{theorem}\label{thm:conservation-1}
$\WKL + \RT^2_2 + \TT^1$ is $\Pi^0_3$-conservative over $\RCA$.
\end{theorem}

In \cite{CLWY}, it is proved that $\TT^1$ is $\Pi^1_1$-conservative over $\RCA + B\Sigma^0_2 + P\Sigma^0_1$.
Since $P\Sigma^0_1$ is a $\Pi^0_3$ sentence, the following is a direct consequence of Theorem \ref{thm:conservation-1}.

\begin{corollary}\label{cor:TT1-PSigma1}
$\WKL + \RT^2_2 + \TT^1 \not\vdash P\Sigma^0_1$.
\end{corollary}

\bigskip

We end this paper with two questions. The second question  generalizes the longstanding open question for $\mathsf{RT}^2_2$:
\begin{enumerate}
\item  Does $\RT^2_2$  imply  $\TT^1$ over $\RCA$?

\vskip.15in

\item Is $\RCA +\RT^2_2 +\TT^1$ a $\Pi^1_1$-conservative system over $\RCA+B\Sigma_2$?
\end{enumerate}

\bibliographystyle{plain}
\bibliography{TT1-consv-main}
\end{document}